\documentclass[12pt,a4paper,reqno]{amsart}
\usepackage[utf8]{inputenc}
\usepackage{enumerate}
\usepackage{mathtools}
\usepackage{mathrsfs}
\usepackage{soul, xcolor}
\usepackage{amsmath,amssymb,xpatch,relsize,graphics,graphicx,amsthm}
\usepackage{float,graphicx}
\usepackage{subcaption}

\xapptocmd\normalsize{%
	\abovedisplayskip=12pt plus 3pt minus 9pt
	\abovedisplayshortskip=0pt plus 3pt
	\belowdisplayskip=12pt plus 3pt minus 9pt
	\belowdisplayshortskip=7pt plus 3pt minus 4pt
}{}{}
\theoremstyle{definition}
\newtheorem{definition}{Definition}[section]

\theoremstyle{plain}
\newtheorem{theorem}[definition]{Theorem}
\newtheorem{corollary}[definition]{Corollary}
\newtheorem{lemma}[definition]{Lemma}

\numberwithin{equation}{section}

\title[Differential Subordination implications............]{\textbf{ Differential Subordination implications for Certain Carath\'{e}odory functions}}
\author[M. Sharma]{Meghna Sharma}
\address{Department of Mathematics, University of Delhi, Delhi--110 007, India}
\email{meghnasharma203@gmail.com}

\author [S. Kumar]{Sushil Kumar}
\address {Bharati Vidyapeeth's college of Engineering, Delhi-110063, India}
\email{sushilkumar16n@gmail.com}

\author [N.K. Jain]{Naveen Kumar Jain}
\address {Department of Mathematics, Aryabhatta College, Delhi-110021,India}
\email{naveenjain05@gmail.com}

\date{}



\thanks{The first author is supported by Junior Research Fellowship from Council of Scientific and Industrial Research, New Delhi, Ref. No.:1753/(CSIR-UGC NET JUNE, 2018).}

\allowdisplaybreaks

\begin{document}
\maketitle

\begin{abstract}
In this article, we wish to establish some
first order differential subordination relations for certain Carath\'{e}odory functions with nice geometrical properties. Moreover, several implications are determined so that the normalized analytic function  belongs to various subclasses of starlike functions.
\end{abstract}

\section{Introduction}
Denote the collection of all  functions $f$ which are analytic on the open unit disc by $\mathscr{H}$.
Let $\mathcal{A} \subset \mathscr{H}$ be the subclass consisting of analytic functions given by $f(z)=z+\sum\limits_{n=2}^{\infty}a_{n}z^{n}$ and normalised by the conditions $f(0)=0$ and $f'(0)-1=0$.
Further, let $\mathcal{S}^{*}$ and $\mathcal{C}$ denote the subclasses of  univalent function consisting of starlike and convex functions, characterised by the quantities $zf'(z)/f(z)$ and $1+zf''(z)/f'(z)$ lying in the interior of the right half plane respectively.
Let $f$ and $g$ be members of $\mathscr{H}$.
We say $f$ is subordinate to $g$ (written as $f \prec g$) if there exists a function $w \in \mathscr{H}$ with $w(0)=0$ and $|w(z)|<1$ such that $f(z) = g(w(z))$.
Equivalently, if $g$ is univalent in $\mathbb{D}$, then the conditions $f(0)=g(0)$ and $f(\mathbb{D}) \subset g(\mathbb{D})$ together gives $f \prec g$. For more details, see~\cite{MR0783572}.
The unified class of starlike functions
$\mathcal{S}^{*}_{\varphi}:=\left\{f \in \mathcal{A}:{zf'(z)}/{f(z)} \prec \varphi(z); \,\,\mbox{for all} \,\, z\in \mathbb{D}\right\}$
where $\varphi$ is analytic, univalent, $\varphi(\mathbb{D})$ is starlike with respect to $\varphi(0)=1$ and $\operatorname{Re}(\varphi)>0$ was introduced and studied by Ma and Minda \cite{MR1343506}. Various subclasses of starlike functions have been studied by considering different choices of $\varphi$ in recent years. For $\varphi(z):=(1+Az)/(1+Bz),(-1 \leq B < A \leq 1)$, the class  $\mathcal{S}^{*}_{\varphi}$ reduces to the class $\mathcal{S}^{*}[A,B]$, introduced by Janowski \cite{MR0267103}.
A function $f \in \mathscr{H}$ is said to be a Carath\'{e}odory function if
$f(0) = 1$ and $\operatorname{Re}(f(z)) > 0$. The class of such functions is denoted by $\mathcal{P}$.
On taking some Carath\'{e}odory functions $\varphi(z):=e^z,\,\phi_q(z),\,\phi_{0}(z), \phi_{c}(z),\,\phi_{\lim}(z),\,\mathcal{Q}(z),\,\phi_{SG},$ $\phi_{s}(z)$, the class $\mathcal{S}^{*}_{\varphi}$ reduce to subclasses
$\mathcal{S}^*_{e}$~\cite{MR3394060}, $\mathcal{S}^{*}_{q}$~\cite{MR3469339}, $\mathcal{S}^{*}_{R}$~\cite{MR3496681}, $\mathcal{S}^{*}_{c}$~\cite{MR3369109}, $\mathcal{S}^{*}_{LC}$~\cite{limc}, $\mathcal{S}^{*}_{B}$~\cite{bell s kumar}, $\mathcal{S}^{*}_{SG}$~\cite{MR4044913}, $\mathcal{S}^{*}_{s}$~\cite{MR3913990} respectively,
where
\begin{align*}
&\phi_q(z)=z+\sqrt{1+z^{2}},
\phi_{0}(z):=1+\frac{z}{k}.\frac{k+z}{k-z};\quad k=1+\sqrt{2},\\
&\phi_{c}(z):=1+\frac{4z}{3}+\frac{2z^2}{3},
\phi_{\lim}(z):=1+\sqrt{2}z+\frac{z^{2}}{2},
\phi_{s}(z):=1+\sin z.
\end{align*}
Recently, Kumar \emph{et al.}\cite{bell s kumar} introduced and studied differential subordination relations and radius estimates for the class $\mathcal{S}^{*}_{B}:= \mathcal{S}^{*}(\mathcal{Q}(z)$, where
\begin{equation}\label{eqn5}
\mathcal{Q}(z):= e^{e^{z}-1}
 \end{equation}

In 2020, Goel and Kumar \cite{MR4044913} studied the subclass  $\mathcal{S}^{*}_{SG}:= \mathcal{S}^{*}(\phi_{SG})$,
where
\begin{equation}\label{eqn6}
\phi_{SG}(z)=2/(1+e^{-z})\,\, \text{for all z} \,\,\in \mathbb{D}.
\end{equation}
Several first order differential subordination results, radius estimates and coefficient estimates were investigated.
These subclasses of starlike functions are well associated with the right half plane of the complex plane and introduced and discussed by various authors.

In 1989, for $p \in \mathcal{P}$, Nunokawa \emph{et al.\@} \cite{MR0975653} proved that the differential subordination $1+ z p'(z) \prec 1+z$ implies $p(z) \prec 1+z$. Further, authors~\cite{Nuno99} established sufficient conditions for starlike functions discussed by Silverman~\cite{Silverman99} to be strongly convex and strongly starlike in $\mathbb{D}$. In 2006, Kanas~\cite{Kanas06} determined the conditions for the functions to map $\mathbb{D}$ onto hyperbolic and parabolic regions using the concept of differential subordination.
In 2007, Ali \emph{et al.\@} \cite{MR2336133} obtained conditions on $\beta \in \mathbb{R}$ for which
$1 + \beta z p'(z)/p^{j}(z) \prec (1+Dz)/(1+Ez)$, $j=0,1,2$ implies
$p(z) \prec (1+Az)/(1+Bz)$, where $A,B,D,E \in [-1,1]$. Later, Kumar and Ravichandran \cite{MR3800966} determined sharp upper bounds on $\beta$ such that
$1 + \beta z p'(z)/p^{j}(z),j=0,1,2$ is subordinate to some Carath\'{e}odory functions like $e^{z}, \phi_{0}(z)$ etc. implies
$p(z) \prec e^{z}$ and $(1+Az)/(1+Bz)$.
For more such results, we refer \cite{Ahuja18,Bohra19,Gandhi18,Ebadian20}.

In the present paper, we determine sharp estimate on $\beta$ so that $p(z) \prec \phi_{q}(z)$, $\mathcal{Q}(z)$, $\phi_{c}(z)$, $\phi_{0}(z)$, $\phi_{\lim}(z)$, $\phi_{s}(z)$, $\phi_{SG}(z)$ whenever $1 + \beta {zp'(z)}/{p^{j}(z)} \prec \mathcal{Q}(z)$ and $\phi_{SG}(z);$ $(j=0,1,2)$.
Further the best possible bound on $\beta$ is computed such that $p(z) \prec \mathcal{Q}(z)$ whenever $1 + \beta {zp'(z)}/{p^{j}(z)} \prec \phi_{c}(z);$ $(j=0,1,2)$.
At last, the upper bound on $\beta$ is estimated so that the subordination
$1 + \beta {zp'(z)}/{p^{j}(z)} \prec \phi_{0}(z)$ and $\phi_{c}(z)$ implies $p(z) \prec \phi_{SG}(z)$.
Moreover, sufficient conditions are obtained for an analytic function $f$ to be a member of a certain subclass of starlike function. \\


\section{Main Results}
Before we prove our main results, we recall following lemma which plays a vital role in our proofs.

\begin{lemma}\label{lemma}\cite[Theorem 3.4h, p.132]{MR1760285}
Let $q : \mathbb{D} \rightarrow \mathbb{C}$ be analytic, and $\psi$ and $v$ be analytic in a domain $U \supseteq q(\mathbb{D})$ with $\psi(w) \neq 0$ whenever $w \in q(\mathbb{D})$. Set
\[Q(z):= zq'(z)\psi(q(z)) \quad \text{and} \quad h(z):=v(q(z))+Q(z),z \in \mathbb{D}.\]
Suppose that
\begin{enumerate}[(i)]
\item either $h(z)$  is convex, or $Q(z)$ is starlike univalent in $\mathbb{D}$ and
\item $\operatorname{Re}\left(\frac{zh'(z)}{Q(z)}\right)>0, z \in \mathbb{D}$.
\end{enumerate}
If $p$ is analytic in $\mathbb{D}$, with $p(0)=q(0)$, $p(\mathbb{D}) \subset U$ and
\[v(p(z))+zp'(z)\psi(p(z)) \prec v(q(z))+zq'(z)\psi(q(z))\]
then $p \prec q$, and $q$ is the best dominant.
\end{lemma}

\noindent
Throughout this paper, the following notations will be used:
\[\Psi_{\beta}(z,p(z)) = 1+\beta zp'(z), \qquad \Lambda_{\beta}(z,p(z)) = 1+\beta \frac{zp'(z)}{p(z)},\quad \mbox{and}\quad \Theta_{\beta}(z,p(z)) = 1+\beta \frac{zp'(z)}{p^{2}(z)}.\]


\begin{theorem}\label{first thm for bell}
Let $\mathcal{Q}(z) \in \mathcal{P}$ be defined by \eqref{eqn5} and further
\begin{equation}\label{eq2}
\mathcal{L}=\int_{-1}^{0}\frac{e^{e^{t}-1}-1}{t}dt \quad{and} \quad \mathfrak{U}=\int_{0}^{1}\frac{e^{e^{t}-1}-1}{t}dt.
\end{equation}
Assume \textit{p} to be an analytic function in $\mathbb{D}$ with $p(0)=1$.
If $\Psi_{\beta}(z,p(z)) \prec \mathcal{Q}(z)$, then
	
\begin{enumerate}[(a)]
\item $p(z) \prec \phi_{q}(z)$ for
$ \beta \geq  \frac{1}{\sqrt{2}} \mathfrak{U} \approx 1.49762$.

\item  $p(z) \prec \mathcal{Q}(z)$ for $\beta \geq  \frac{1}{1-e^{(e^{-1}-1)}} \mathcal{L} \approx 1.446103$.

\item $p(z) \prec \phi_{c}(z)$ for $\beta \geq \frac{1}{2}\mathfrak{U} \approx 1.05898$.
		
\item  $p(z) \prec \phi_{0}(z)$ for $\beta \geq \left(3+2\sqrt{2}\right)\mathcal{L} \approx 3.94906$.

\item  $p(z) \prec \phi_{\lim}(z)$ for $\beta \geq  \frac{2}{2\sqrt{2}+1} \mathfrak{U} \approx 1.10643$.

\item  $p(z) \prec \phi_{s}(z)$ for $\beta \geq  \frac{1}{\sin1} \mathfrak{U} \approx 2.51696$.

\item  $p(z) \prec \phi_{SG}(z)$ for $\beta \geq  \frac{e+1}{e-1} \mathfrak{U} \approx 4.583145$.
\end{enumerate}
The bounds in each case are sharp.
\end{theorem}

\begin{proof}
The analytic function
$q_{\beta}:\overline{\mathbb{D}} \rightarrow \mathbb{C}$ defined by
\[q_{\beta}(z)=1+\frac{1}{\beta}\int_{0}^{z}\frac{e^{e^{t}-1}-1}{t}dt\]
is a solution of the first order linear differential equation
$1+\beta zq'_{\beta}(z) = e^{e^{z}-1}.$
For $w \in \mathbb{C}$, define the functions $v(w)=1$ and $\psi(w)=\beta$.
Now, the function $Q:\overline{\mathbb{D}} \rightarrow \mathbb{C}$ defined by
\[Q(z)=zq'_{\beta}(z)\psi(q_{\beta}(z)) = \beta zq'_{\beta}(z) = e^{e^{z}-1}-1\]
is starlike in $\mathbb{D}$.
Also, note that by analytic characterization of starlike functions, the function $h:\overline{\mathbb{D}} \rightarrow \mathbb{C}$ defined by
$h(z):= v(q_{\beta}(z))+Q(z)$
satisfies the inequality \[\operatorname{Re}\left(\frac{zh'(z)}{Q(z)}\right)=\operatorname{Re}\left(\frac{zQ'(z)}{Q(z)}\right)>0.\]
Therefore, the subordination
$1+\beta zp'(z) \prec 1+\beta zq'_{\beta}(z)$ implies $p \prec q_{\beta}$ by Lemma \ref{lemma}.
 For suitable  $\mathcal{P}(z)$, as $r \to 1$, $q_{\beta}(z) \prec \mathcal{P}(z)$ holds if the following inequalities holds:
\begin{equation}\label{sufficient}
\mathcal{P}(-1) < q_{\beta}(-1) < q_{\beta}(1) < \mathcal{P}(1).
\end{equation}
By the transitivity property, the required subordination $p(z) \prec \mathcal{P}(z)$ holds  if $q_{\beta}(z) \prec \mathcal{P}(z)$.
The condition \eqref{sufficient} turns out to be both necessary and sufficient for the subordination $p \prec \mathcal{P}$ to hold.
	
\begin{enumerate}[(a)]
\item Consider $\mathcal{P}(z)=\phi_{q}(z)$. Then the inequalities $q_{\beta}(-1) > -1+\sqrt{2}$ and $q_{\beta}(1) < 1+\sqrt{2}$ reduce to $\beta \geq \beta_{1}$ and $\beta \geq \beta_{2}$, where
\[\beta_{1}=\frac{1}{2-\sqrt{2}} \mathcal{L} \quad \text{and} \quad \beta_{2}=\frac{1}{\sqrt{2}} \mathfrak{U}\]
respectively. Thus, the subordination $q_{\beta} \prec \phi_{q}$ holds whenever
$\beta \geq \max\{\beta_{1},\beta_{2}\} = \beta_{2}$.
	
\item For $\mathcal{P}(z)=\mathcal{Q}(z)$, the inequalities
$q_{\beta}(-1) > \mathcal{Q}(-1)$ and $q_{\beta}(1) < \mathcal{Q}(1)$ give $\beta \geq \beta_{1}$ and $\beta \geq \beta_{2}$, where
\[\beta_{1}=\frac{1}{1-e^{e^{-1}-1}} \mathcal{L} \quad \text{and} \quad \beta_{2}=\frac{1}{e^{e-1}-1} \mathfrak{U}\]
respectively. Therefore, $q_{\beta} \prec \mathcal{Q}$ whenever
$\beta \geq \max\{\beta_{1},\beta_{2}\} = \beta_{1}$.

\item On taking $\mathcal{P}(z)=\phi_{c}(z)$, a simple calculation shows that the inequalities $q_{\beta}(-1) > \phi_{c}(-1)$ and $q_{\beta}(1) < \phi_{c}(1)$ give $\beta \geq \beta_{1}$ and $\beta \geq \beta_{2}$, where
$\beta_{1}=({3}/{2}) \mathcal{L}$ and $\beta_{2}=({1}/{2}) \mathfrak{U}$
respectively. Therefore, $q_{\beta} \prec \phi_{c}$ holds whenever
$\beta \geq \max\{\beta_{1},\beta_{2}\} = \beta_{2}$.
		
\item On substituting $\mathcal{P}(z)=\phi_{0}(z)$, the inequalities $q_{\beta}(-1) > \phi_{0}(-1)$ and $q_{\beta}(1) < \phi_{0}(1)$ give $\beta \geq \beta_{1}$ and $\beta \geq \beta_{2}$, where
$\beta_{1}=(1/(3-2\sqrt{2})) \mathcal{L} $ and $ \beta_{2}= \mathfrak{U}$
respectively. Therefore, the subordination $q_{\beta} \prec \phi_{0}$ holds if
$\beta \geq \max\{\beta_{1},\beta_{2}\}=\beta_{1}$.

\item Take $\mathcal{P}(z)=\phi_{\lim}(z)$.
Then the inequalities $q_{\beta}(-1) > \frac{3}{2}-\sqrt{2}$ and $q_{\beta}(1) < \frac{3}{2}+\sqrt{2}$ reduce to $\beta \geq \beta_{1}$ and $\beta \geq \beta_{2}$, where
$\beta_{1}=(2/(2\sqrt{2}-1) \mathcal{L}$ and $ \beta_{2}=2/(2\sqrt{2}+1) \mathfrak{U}$
respectively. Thus, the required subordination $q_{\beta} \prec \phi_{\lim}$ holds  if
$\beta \geq \max\{\beta_{1},\beta_{2}\} = \beta_{2}$.
		
\item Take $\mathcal{P}(z)=\phi_{s}(z)$. Then the inequalities $q_{\beta}(-1) > 1+\sin(-1)$ and $q_{\beta}(1) < 1+\sin(1)$ give $\beta \geq \beta_{1}$ and $\beta \geq \beta_{2}$, where
$\beta_{1}={\mathcal{L}}/{\sin1}$ and $\beta_{2}={\mathfrak{U}}/{\sin1}$
respectively. This shows that  the subordination $q_{\beta} \prec \phi_{s}$ holds if
$\beta \geq \max\{\beta_{1},\beta_{2}\} = \beta_{2}$.
		
\item Set $\mathcal{P}(z) = 2/(1+e^{-z})$. Then $q_{\beta}(-1) > 2/(e+1)$ and $ q_{\beta}(1) < 2e/(e+1)$ gives
\[\beta_{1} = \frac{e+1}{e-1} \mathcal{L} \quad \text{and} \quad \beta_{2} = \frac{e+1}{e-1} \mathfrak{U}.\]
Hence, the subordination holds true for $\beta \geq \beta_{2}$ since $\max\{\beta_{1},\beta_{2}\} = \beta_{2}$.
\end{enumerate}
\begin{figure}[h!]
	\centering
	\begin{subfigure}[!]{0.3\linewidth}
		\includegraphics[width=\linewidth]{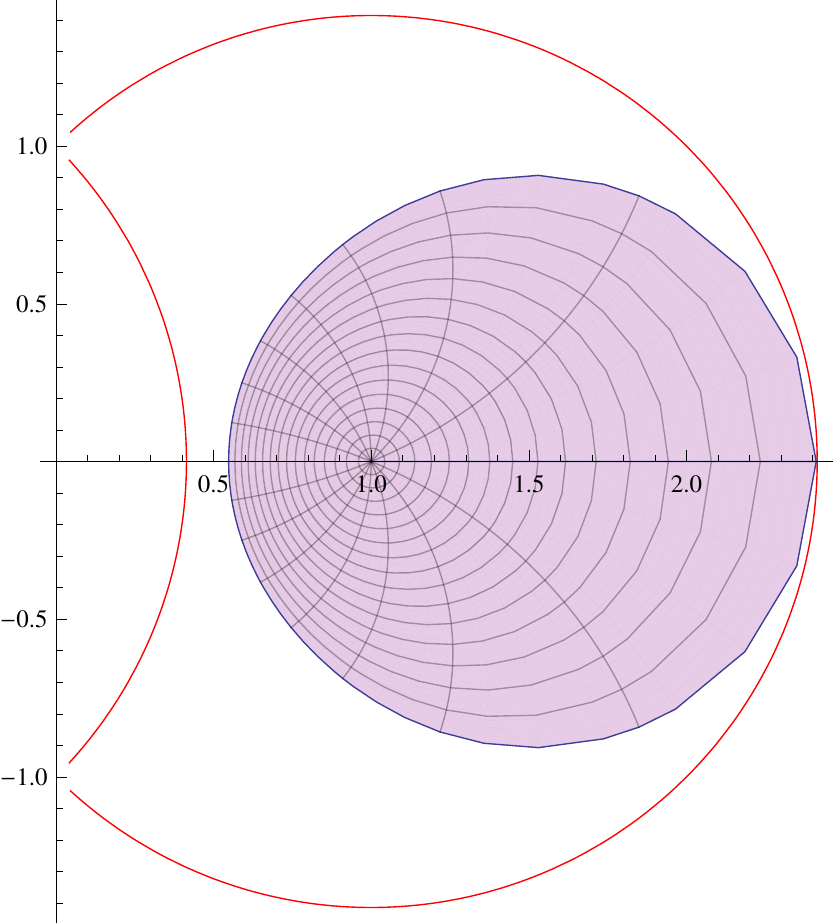}
	\end{subfigure}
\hspace{4em}
	\begin{subfigure}[!]{0.3\linewidth}
		\includegraphics[width=\linewidth]{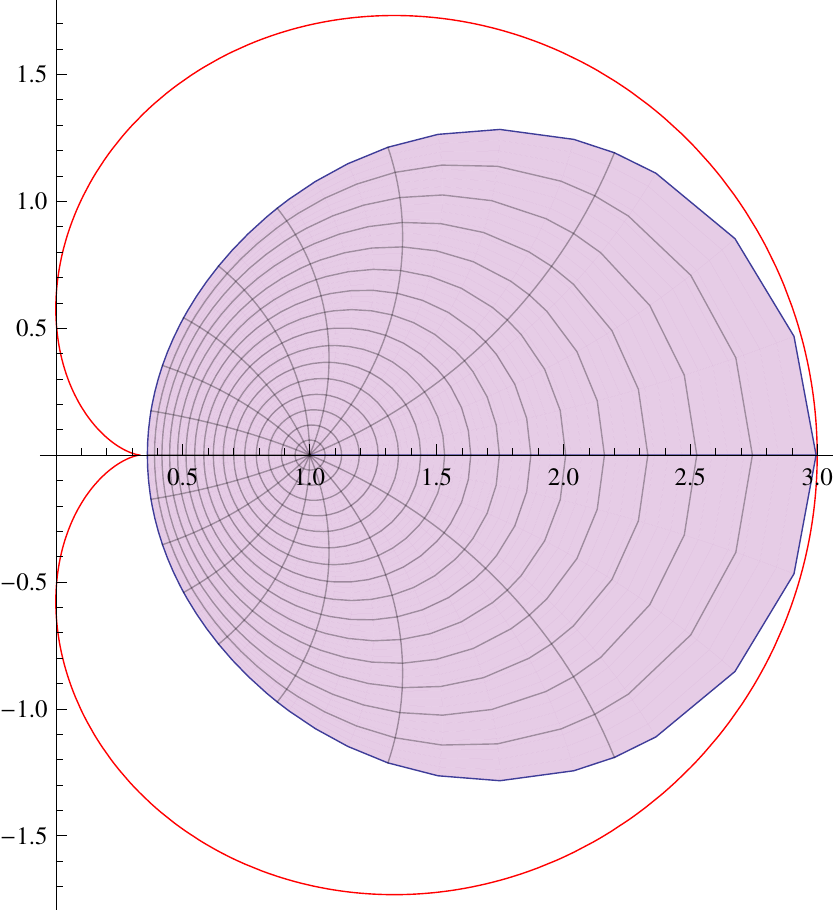}
	\end{subfigure}
	\caption{Sharpness for the case (a) and (b).}
	\label{fig:coffee}
\end{figure}
\qedhere
\end{proof}

As an application of Theorem \ref{first thm for bell}, we have the following sufficient conditions for starlikeness:
\begin{corollary}\label{Cor1}
Set $\mathfrak{M}(z):= 1 - zf'(z)/f(z) + zf''(z)/f'(z)$.
If the function $f \in \mathcal{A}$ satisfies
$1+ \beta \frac{zf'(z)}{f(z)}\mathfrak{M}(z) \prec \mathcal{Q}(z)$,
then
\begin{itemize}
\item [(a)]$f \in \mathcal{S}^{*}_{q}$ if
$ \beta \geq \left({1}/{\sqrt{2}}\right)\mathfrak{U}$,
\item [(b)] $f \in \mathcal{S}^{*}_{B}$ if
$ \beta \geq \left({1}/{(1-e^{(e^{-1}-1)})}\right)\mathcal{L}$,
\item [(c)] $f \in \mathcal{S}^{*}_{c}$ if
$\beta \geq \left({1}/{2}\right)\mathfrak{U}$,
\item [(d)] $f \in \mathcal{S}^{*}_{R}$ if
$ \beta \geq \left(3+2\sqrt{2}\right)\mathcal{L}$,
\item [(e)] $f \in \mathcal{S}^{*}_{LC}$ if
$ \beta \geq \left({2}/{(2\sqrt{2}+1)}\right)\mathfrak{U}$,
\item [(f)] $f \in \mathcal{S}^{*}_{s}$ if
$ \beta \geq \left({1}/{(\sin1)}\right)\mathfrak{U}$
\item [(g)] $f \in \mathcal{S}^{*}_{SG}$ if
$ \beta \geq \left({(e+1)}/{(e-1)}\right)\mathfrak{U}$,
\end{itemize}
where $\mathfrak{U}$ and $\mathcal{L}$  are given by \eqref{eq2}.
\end{corollary}

\begin{theorem}\label{th1}
Let $\mathfrak{U}$ and $\mathcal{L}$  be given by \eqref{eq2} and $\mathcal{Q}(z)$ be given by \eqref{eqn5}.
Let $p$ be an analytic function in $\mathbb{D}$ with $p(0)=1$.
If $\Lambda_{\beta}(z,p(z)) \prec \mathcal{Q}(z)$, then
		
\begin{enumerate}[(a)]
\item $p(z) \prec \phi_{q}(z)$ for $ \beta \geq \frac{1}{\log(1+\sqrt{2})}\mathfrak{U} \approx 2.40301$.

\item  $p(z) \prec \mathcal{Q}(z)$ for $ \beta \geq \frac{1}{e-1}\mathfrak{U} \approx 1.23260$.

\item $p(z) \prec \phi_{c}(z)$ for $\beta \geq  \frac{1}{\log3} \mathfrak{U} \approx 1.92784$.
			
\item  $p(z) \prec \phi_{0}(z)$ for $ \beta \geq  \frac{1}{\log\left(\frac{1+\sqrt{2}}{2}\right)} \mathcal{L} \approx 3.59966$.

\item  $p(z) \prec \phi_{\lim}(z)$ for $ \beta \geq  \frac{1}{\log(\sqrt{2}+3/2)} \mathfrak{U} \approx 1.98013$.
			
\item  $p(z) \prec \phi_{s}(z)$ for $ \beta \geq  \frac{1}{\log(1+\sin1)} \mathfrak{U} \approx 3.4688$.	
			
\item  $p(z) \prec \phi_{SG}(z)$ for $ \beta \geq  \frac{1}{1+\log2-\log(1+e)} \mathfrak{U} \approx 5.57523$.
\end{enumerate}
The bounds on $\beta$ are best possible.
\end{theorem}
	
\begin{proof}
Consider the first order differential equation given by
\begin{equation}\label{diff}
1+\beta \frac{z\breve{q}_{\beta}'(z)}{\breve{q}_{\beta}(z)} = e^{e^{z}-1}.
\end{equation}
It is easy to verify that the analytic function
$\breve{q}_{\beta} : \overline{\mathbb{D}} \rightarrow \mathbb{C}$
defined by
\[\breve{q}_{\beta}(z)=\exp\left(\frac{1}{\beta}\int_{0}^{z}\frac{e^{e^{t}-1}-1}{t}dt\right)\]
is a solution of differential equation  \eqref{diff}.
On taking $v(w)=1$ and $\psi(w)=\beta /w$, the functions $Q,h : \overline{\mathbb{D}} \rightarrow \mathbb{C}$
reduces to
\[Q(z)=z\breve{q}'_{\beta}(z)\psi(\breve{q}_{\beta}(z)) = \beta z\breve{q}'_{\beta}(z)/\breve{q}_{\beta}(z) = e^{e^{z}-1}-1\]
and
\[h(z) = v(\breve{q}_{\beta}(z))+Q(z) = 1+Q(z) = e^{e^{z}-1}.\]
It is seen that the function $Q$ is starlike and $\operatorname{Re}\left({zh'(z)}/{Q(z)}\right)>0,$ for $z\in \mathbb{D}$.
Hence,
\[1+\beta \frac{zp'(z)}{p(z)} \prec 1+\beta \frac{z\breve{q}'_{\beta}(z)}{\breve{q}_{\beta}(z)}\quad \text{implies}\quad p(z) \prec \breve{q}_{\beta}(z)\]
which follows from Lemma \ref{lemma}.
Proceeding as Theorem in \ref{first thm for bell}, proof is completed.
	\end{proof}
	
	
%
%
		
\noindent
\begin{theorem}\label{th2}
Let $\mathfrak{U}$ and $\mathcal{L}$ be given by \eqref{eq2} and $\mathcal{Q}(z)$ be given by \eqref{eqn5}.
Assume $p$ to be an analytic function in $\mathbb{D}$ with $p(0)=1$.
If $\Theta_{\beta}(z,p(z)) \prec \mathcal{Q}(z)$, then each of the following subordination holds:
\begin{enumerate}[(a)]
\item $p(z) \prec \phi_{q}(z)$ for $ \beta \geq \frac{1}{2-\sqrt{2}}\mathfrak{U} \approx 3.61556 $.

\item  $p(z) \prec \mathcal{Q}(z)$ for $ \beta \geq \frac{e^{e-1}}{e^{e-1}-1}\mathfrak{U} \approx 2.58089$.

\item $p(z) \prec \phi_{c}(z)$ for $\beta \geq \frac{3}{2}\mathfrak{U} \approx 3.17692 $.
				
\item  $p(z) \prec \phi_{0}(z)$ for $ \beta \geq 2\mathfrak{U} \approx 4.2359$.

\item  $p(z) \prec \phi_{\lim}(z)$ for $ \beta \geq \frac{5+4\sqrt{2}}{7}\mathfrak{U} \approx 3.22438$.
				
\item  $p(z) \prec \phi_{s}(z)$ for $ \beta \geq \frac{1+\sin1}{\sin1}\mathfrak{U} \approx 4.63491$.
				
\item  $p(z) \prec \phi_{SG}(z)$ for $ \beta \geq \frac{2e}{e-1}\mathfrak{U} \approx 6.7011$.	
\end{enumerate}
The estimates on $\beta$ cannot be improved further.
\end{theorem}
		
\begin{proof}
The function
\[\hat{q}_{\beta}(z)=\left(1-\frac{1}{\beta}\int_{0}^{z}\frac{e^{e^{t}-1}-1}{t}dt\right)^{-1}\]
is the analytic solution of the differential equation
\[\beta \frac{z\hat{q}'_{\beta}(z)}{\hat{q}_{\beta}^{2}(z)} = e^{e^{z}-1}-1.\]
Consider the functions $v(w)=1$ and $\psi(w)=\beta/w^2$.
Moreover, the function
$Q(z) = z\hat{q}'_{\beta}(z) \psi(\hat{q}_{\beta}(z)) =  e^{e^{z}-1}-1$ is starlike in $\mathbb{D}$.
Simple computation shows that the function
$h(z):= 1+Q(z)$ satisfies the inequality
$\operatorname{Re}\left({zh'(z)}/{Q(z)}\right)>0, (z \in \mathbb{D})$.
Now, by Lemma \ref{lemma}, we see that the subordination
\[1+\beta \frac{zp'(z)}{p^{2}(z)} \prec 1+\beta \frac{z\hat{q}'_{\beta}(z)}{\hat{q}^{2}_{\beta}(z)}\]
implies $p(z) \prec \hat{q}_{\beta}(z)$.
Proceeding as in Theorem \ref{first thm for bell}, we conclude the proof.
\end{proof}
	

\begin{theorem}\label{theo2.6}
Let $\phi_{SG}$ be given by \eqref{eqn6} and further
\begin{equation}\label{eq3}
I_{-}=\int_{-1}^{0}\frac{e^{t}-1}{t(e^{t}+1)}dt \quad\text{and} \quad I_{+}=\int_{0}^{1}\frac{e^{t}-1}{t(e^{t}+1)}dt.
\end{equation}
Assume \textit{p} to be an analytic function in $\mathbb{D}$ with $p(0)=1$.
If the  subordination
\[\Psi_{\beta}(z,p(z)) \prec \phi_{SG}(z)\]
holds, then each of the following subordination inclusion hold:
\begin{enumerate}[(a)]
\item $p(z) \prec \phi_{q}(z)$ for
$\beta \geq \frac{1}{2-\sqrt{2}} I_{-} \approx 0.83117$.
			
\item $p(z) \prec \phi_{c}(z)$ for
$\beta \geq \frac{3}{2}I_{-} \approx 0.730335$.
			
\item  $p(z) \prec \phi_{0}(z)$ for
$\beta \geq (3+2\sqrt{2})I_{-} \approx 2.837797$.
			
\item  $p(z) \prec \mathcal{Q}(z)$ for
$\beta \geq \frac{1}{1-e^{e^{-1}-1}}I_{-} \approx 1.039170$.
			
\item  $p(z) \prec \phi_{\lim}(z)$ for
$\beta \geq \frac{2}{2\sqrt{2}-1}I_{-} \approx 0.53257$.
			
\item  $p(z) \prec \phi_{s}(z)$ for $ \beta \geq \frac{1}{\sin1}I_{-} \approx 0.578616$

\item  $p(z) \prec \phi_{SG}(z)$ for $ \beta \geq \frac{e+1}{e-1}I_{-} \approx 1.05361$.	
\end{enumerate}
The bounds in each of the above case are sharp.
\end{theorem}

\begin{proof}
Consider the functions $v$ and $\psi$ defined as in Theorem \ref{first thm for bell}.
Define the function $q_{\beta}:\overline{\mathbb{D}} \rightarrow \mathbb{C}$ by
\[q_{\beta}(z)=1+\frac{1}{\beta}\int_{0}^{z}\frac{e^{t}-1}{t(e^{t}+1)}dt\]
Note that the function $q_{\beta}(z)$ is analytic solution of the differential equation
$1+\beta z q'_{\beta}(z)={2}/({1+e^{-z}}).$
The function $Q(z)=zq_{\beta}'(z)\psi(q_{\beta}(z))={(e^{z}-1)}/{(e^{z}+1)}$
is starlike in $\mathbb{D}$ and $h(z)=1+Q(z)$ satisfies the inequality $\operatorname{Re}\left({zh'(z)}/{Q(z)}\right)>0,$  $z\in \mathbb{D}$.
Thus, applying Lemma \ref{lemma}, it follows that the subordination
$1+\beta zp'(z) \prec 1+\beta zq'_{\beta}(z)$ implies $p(z) \prec q_{\beta}(z)$.
Each of the subordination $p(z) \prec \mathcal{P}(z)$, for appropriate $\mathcal{P}$, from (a) to (g) holds if $q_{\beta}(z) \prec \mathcal{P}(z)$ holds. This subordination holds provided
\[\mathcal{P}(-1) < q_{\beta}(-1) < q_{\beta}(1) < \mathcal{P}(1).\]
These inequalities yield necessary and sufficient condition for the required subordination.
		
\begin{enumerate}[(a)]
\item Take $\mathcal{P}(z)=\phi_{q}(z)$. Then, the inequalities $q_{\beta}(-1) > -1+\sqrt{2}$ and $q_{\beta}(1) < 1+\sqrt{2}$ reduce to $\beta \geq \beta_{1}$ and $\beta \geq \beta_{2}$, where
\[\beta_{1}=\frac{1}{2-\sqrt{2}} I_{-} \quad \text{and} \quad \beta_{2}=\frac{1}{\sqrt{2}} I_{+}\]
respectively. Therefore,  $q_{\beta} \prec \phi_{q}$ whenever
$\beta \geq \max\{\beta_{1},\beta_{2}\} = \beta_{1}$.
			
\item Consider $\mathcal{P}(z)=\phi_{c}(z)$. A simple calculation shows that the inequalities $q_{\beta}(-1) > \phi_{c}(-1)$ and $q_{\beta}(1) < \phi_{c}(1)$ gives $\beta \geq \beta_{1}$ and $\beta \geq \beta_{2}$, where
\[\beta_{1}=\frac{3}{2} I_{-} \quad \text{and} \quad\beta_{2}=\frac{1}{2} I_{+}\]
respectively. Therefore, the subordination $q_{\beta} \prec \phi_{c}$ holds if
$\beta \geq \max\{\beta_{1},\beta_{2}\} = \beta_{1}$.
			
\item On taking $\mathcal{P}(z)=\phi_{0}(z)$, the inequalities $q_{\beta}(-1) > \phi_{0}(-1)$ and $q_{\beta}(1) < \phi_{0}(1)$ give $\beta \geq \beta_{1}$ and $\beta \geq \beta_{2}$, where
\[\beta_{1}=\frac{1}{3-2\sqrt{2}} I_{-} \quad \text{and} \quad \beta_{2}= I_{+}\]
respectively. Therefore,
$q_{\beta} \prec \phi_{0}$  if
$\beta \geq \max\{\beta_{1},\beta_{2}\} = \beta_{1}$.
			
\item Consider $\mathcal{P}(z)=\mathcal{Q}(z)$. From the inequalities $q_{\beta}(-1) > \mathcal{Q}(-1)$ and $q_{\beta}(1) < \mathcal{Q}(1)$,
we note that $\beta \geq \beta_{1}$ and $\beta \geq \beta_{2}$, where
\[\beta_{1}=\frac{1}{1-e^{e^{-1}-1}} I_{-} \quad \text{and} \quad \beta_{2}=\frac{1}{e^{e-1}-1} I_{+}\]
respectively.
Thus, $q_{\beta} \prec \mathcal{Q}$ if $\beta \geq \max\{\beta_{1},\beta_{2}\}$.
			
\item Take $\mathcal{P}(z)=\phi_{\lim}(z)$. Then, the inequalities $q_{\beta}(-1) > \frac{3}{2}-\sqrt{2}$ and $q_{\beta}(1) < \frac{3}{2}+\sqrt{2}$ reduce to $\beta \geq \beta_{1}$ and $\beta \geq \beta_{2}$, where
\[\beta_{1}=\frac{2}{2\sqrt{2}-1} I_{-} \quad \text{and} \quad \beta_{2}=\frac{2}{2\sqrt{2}+1} I_{+}\]
respectively. Thus,  $q_{\beta} \prec \phi_{\lim}$ whenever
$\beta \geq \max\{\beta_{1},\beta_{2}\} = \beta_{1}$.
			
\item Take $\mathcal{P}(z)=\phi_{s}(z)$. Then, the inequalities $q_{\beta}(-1) > 1+\sin(-1)$ and $q_{\beta}(1) < 1+\sin(1)$ give $\beta \geq \beta_{1}$ and $\beta \geq \beta_{2}$, where
\[\beta_{1}=\frac{1}{\sin1} I_{-} \quad \text{and} \quad\beta_{2}=\frac{1}{\sin1} I_{+}\]
respectively. Therefore, the subordination $q_{\beta} \prec \phi_{s}$ holds if
$\beta \geq \max\{\beta_{1},\beta_{2}\} = \beta_{1}$.

\item Let $\mathcal{P}(z)=\phi_{SG}(z)$. On simplifying the inequalities $q_{\beta}(-1)>2/(e+1)$ and $q_{\beta}(1)<2e/(e+1)$, we get $\beta_{1}$ and $\beta_{2}$, where
\[\beta_{1}=\frac{e+1}{e-1} I_{-} \quad \text{and} \quad\beta_{2}=\frac{e+1}{e-1} I_{+}\]
respectively and thus, $q_{\beta} \prec \phi_{SG}$ whenever $\beta \geq \max\{\beta_{1},\beta_{2}\} = \beta_{1}$.
\end{enumerate}
\begin{figure}[h!]
	\centering
	\begin{subfigure}[!]{0.3\linewidth}
		\includegraphics[width=\linewidth]{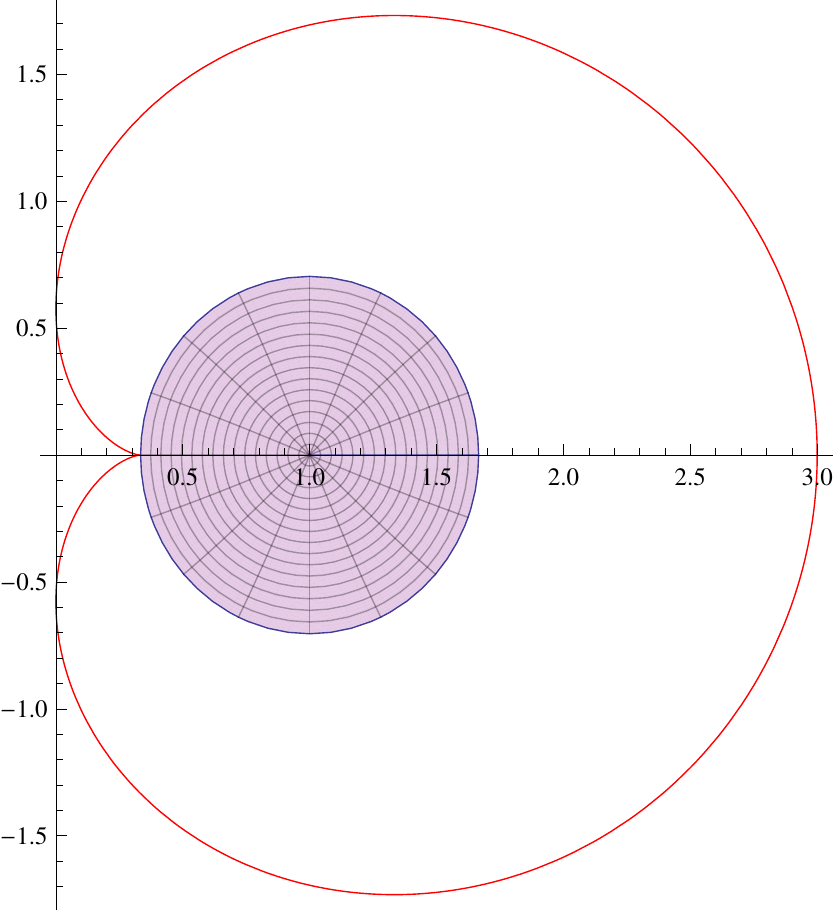}
	\end{subfigure}
	\hspace{4em}
	\begin{subfigure}[!]{0.3\linewidth}
		\includegraphics[width=\linewidth]{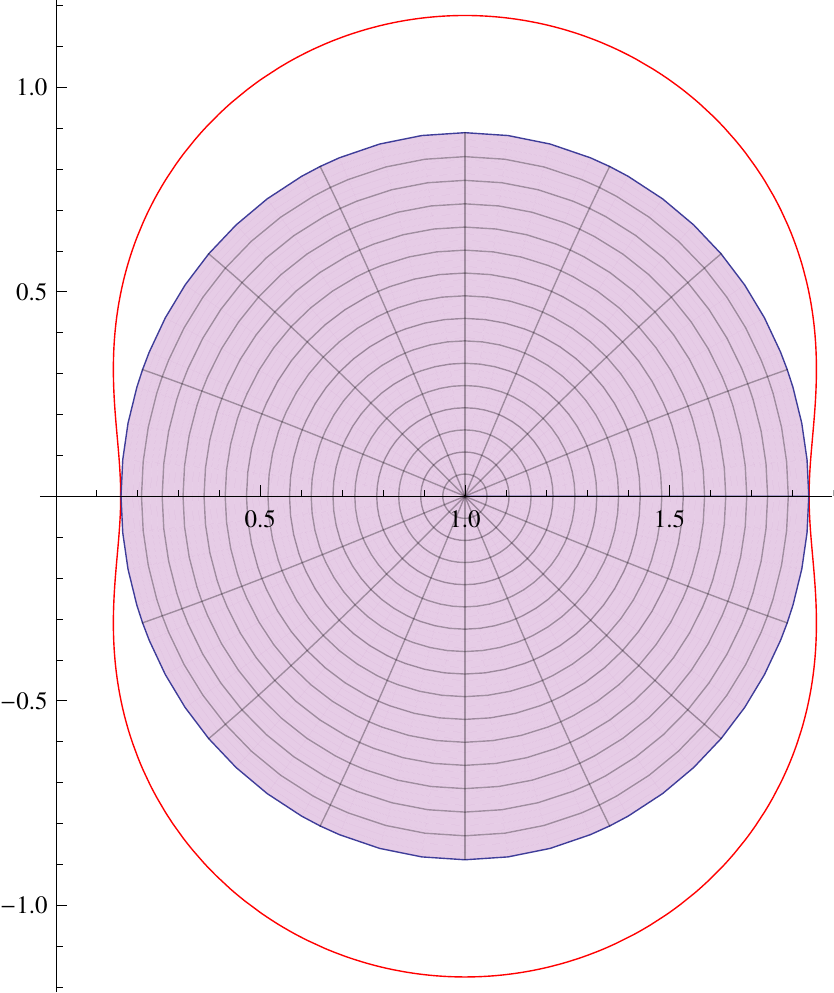}
	\end{subfigure}
	\caption{Sharpness for the case (b) and (f).}
	\label{fig 2}
\end{figure}
\qedhere
\end{proof}
As an application of Theorem \ref{theo2.6}, we have the following sufficient conditions for starlikeness:
\begin{corollary}
Let $f \in \mathcal{A}$ be analytic function which satisfies
\[1+ \beta \frac{zf'(z)}{f(z)} \mathfrak{M}(z) \prec \phi_{SG}(z).\]
Then,
\begin{enumerate}[(a)]
\item $f \in \mathcal{S}^{*}_{q}$ if
$ \beta \geq \left({1}/{(2-\sqrt{2})}\right)I_{-}$,
\item $f \in \mathcal{S}^{*}_{c}$ if
$\beta \geq \left({3}/{2}\right)I_{-}$,
\item $f \in \mathcal{S}^{*}_{R}$ if
$\beta \geq (3+2\sqrt{2})I_{-}$,
\item $f \in \mathcal{S}^{*}_{B}$ if
$\beta \geq \left({1}/{(1-e^{e^{-1}-1})}\right)I_{-}$,
\item $f \in \mathcal{S}^{*}_{LC}$ if
$\beta \geq \left({2}/{(2\sqrt{2}-1)}\right)I_{-}$,
\item $f \in \mathcal{S}^{*}_{s}$ if
$\beta \geq \left({1}/{(\sin1)}\right)I_{-}$,
\end{enumerate}
where $\mathfrak{M}(z)$ is defined in Corollary \ref{Cor1}.
\end{corollary}

\begin{theorem}
Let $I_{+}$ and $I_{-}$ be given by \ref{eq3} and $\phi_{SG}$ be given by \eqref{eqn6}. Assume $p$ to be an analytic function in $\mathbb{D}$ with $p(0)=1$.
If $\Lambda_{\beta}(z,p(z)) \prec \phi_{SG}(z)$, then each of the following holds.
\begin{enumerate}[(a)]
\item $p(z) \prec \phi_{q}(z)$ for $ \beta \geq \frac{1}{\log(1+\sqrt{2})}I_{-} \approx 0.55242$.
		
\item $p(z) \prec \phi_{c}(z)$ for $\beta \geq \frac{1}{\log3}I_{-} \approx 0.443185$.
		
\item  $p(z) \prec \phi_{0}(z)$ for $ \beta \geq \frac{1}{\log\left(\frac{1+\sqrt{2}}{2}\right)}I_{-} \approx 2.58671 $.
		
\item  $p(z) \prec \mathcal{Q}(z)$ for $ \beta \geq \frac{1}{1-e^{-1}}I_{-} \approx 0.77024 $.
		
\item  $p(z) \prec \phi_{\lim}(z)$ for $ \beta \geq \frac{1}{\log(\sqrt{2}+3/2)}I_{+} \approx 0.455206 $.

\item  $p(z) \prec \phi_{s}(z)$ for $ \beta \geq \frac{1}{\log(1+\sin1)}I_{+} \approx 0.79744 $.

\item  $p(z) \prec \phi_{SG}(z)$ for $ \beta \geq \frac{1}{1+\log2-\log(1+e)}I_{+} \approx 1.28167$.	
\end{enumerate}
The estimates on $\beta$ are best possible.
\end{theorem}

\begin{proof}
Let the functions $v$ and $\psi$ be defined as in Theorem \ref{th1}.
Define the analytic function $\breve{q}_{\beta}:\overline{\mathbb{D}} \rightarrow \mathbb{C}$ by
\[\breve{q}_{\beta}(z)=\exp\left(\frac{1}{\beta}\int_{0}^{z}\frac{e^{t}-1}{t(e^{t}+1)}dt\right),\]
which satisfies the differential equation
\[\frac{d\breve{q}_{\beta}'(z)}{dz}=\frac{1}{\beta z}\left(\frac{1-e^{-z}}{1+e^{-z}}\right)\breve{q}_{\beta}(z).\]
Now, observe that the function
$Q(z)=z\breve{q}'_{\beta}(z)\psi(\breve{q}_{\beta}(z))=\frac{1-e^{-z}}{1+e^{-z}}$
is starlike in $\mathbb{D}$.
Also, it can be easily seen that the function $h$ defined by
$h(z):= v(\breve{q}_{\beta}(z))+Q(z) = 1+Q(z)$
satisfies the inequality
$\operatorname{Re}\left({zh'(z)}/{Q(z)}\right)>0,$  $z\in \mathbb{D}$.
Therefore, the Lemma \ref{lemma} states that the subordination
$1+\beta \frac{zp'(z)}{p(z)} \prec 1+\beta \frac{z\breve{q}_{\beta}'(z)}{\breve{q}_{\beta}(z)}$
implies $p(z) \prec \breve{q}_{\beta}(z)$.
%
%
As in the proof of Theorem \ref{theo2.6}, we conclude the result.
\end{proof}

\begin{theorem}
Let $I_{+}$ and $I_{-}$ be given by \eqref{eq3}.
Assume \textit{p} to be an analytic function in $\mathbb{D}$ with $p(0)=1$.
If  $\Theta_{\beta}(z,p(z)) \prec \phi_{SG}(z)$, then
\begin{enumerate}[(a)]
\item $p(z) \prec \phi_{q}(z)$ for $\beta \geq \frac{1}{2-\sqrt{2}}I_{+} \approx 0.83117 $.
		
\item $p(z) \prec \phi_{c}(z)$ for $\beta \geq \frac{3}{2}I_{+} \approx 0.73033 $.
		
\item $p(z) \prec \phi_{0}(z)$ for $\beta \geq (2+2\sqrt{2})I_{-} \approx 2.35090$.
		
\item  $p(z) \prec \mathcal{Q}(z)$ for $\beta \geq\frac{e^{e-1}}{e^{e-1}-1}I_{+} \approx 0.59331$.
		
\item  $p(z) \prec \phi_{\lim}(z)$ for $\beta \geq \frac{5+4\sqrt{2}}{7}I_{+} \approx 0.74124$.
		
\item  $p(z) \prec \phi_{s}(z)$ for $\beta \geq \frac{1+\sin1}{\sin1}I_{+} \approx 1.06550$.

\item  $p(z) \prec \phi_{SG}(z)$ for $\beta \geq \frac{2e}{e-1}I_{+} \approx 1.54049$.		
\end{enumerate}
All these estimates are sharp.
\end{theorem}

\begin{proof}
The function $\hat{q}_{\beta}:\overline{\mathbb{D}} \rightarrow \mathbb{C}$ defined by
\[\hat{q}_{\beta}(z)=\left(1-\frac{1}{\beta}\int_{0}^{z}\frac{e^{t}-1}{t(e^{t}+1)}dt\right)^{-1}\]
is clearly analytic in $\mathbb{D}$.
It is noted that the function $\hat{q}_{\beta}(z)$ is a solution of the differential equation
\[1+\beta \frac{z\hat{q}'_{\beta}(z)}{\hat{q}^{2}_{\beta}(z)}=\frac{2}{1+e^{-z}}.\]
We take the functions $v$ and $\psi$ as in Theorem \ref{th2}.
Note that the function $Q$ defined by
$Q(z)=z\hat{q}'_{\beta}(z)\psi(\hat{q}_{\beta}(z))=(1-e^{-z})/(1+e^{-z})$
is starlike in $\mathbb{D}$ and the function $h$ defined as
$h(z):= v(\hat{q}_{\beta}(z))+Q(z) = 1+Q(z)$ follows the inequality $\operatorname{Re}\left({zh'(z)}/{Q(z)}\right)=\operatorname{Re}\left({zQ'(z)}/{Q(z)}\right)>0$.
Therefore, as in view of Lemma \ref{lemma}, the subordination
\[1+\beta \frac{zp'(z)}{p^{2}(z)} \prec 1+\beta \frac{z\hat{q}'_{\beta}(z)}{\hat{q}_{\beta}^{2}(z)}\]
implies $p(z) \prec \hat{q}_{\beta}(z)$.
Proceeding as in Theorem \ref{theo2.6}, proof is completed.
%
\end{proof}

\begin{theorem}\label{thm1}
Let \textit{p} be an analytic function in $\mathbb{D}$ with $p(0)=1$. Then each of the following subordination implies $p(z) \prec \mathcal{Q}(z):= e^{e^{z}-1}$:
\begin{enumerate}[(a)]
\item $\Psi_{\beta}(z,p(z)) \prec \phi_{c}(z)$ if
$\beta$ $\geq$ $\frac{1}{1-e^{(e^{-1}-1)}} \approx 2.13430$.
\item $\Lambda_{\beta}(z,p(z)) \prec \phi_{c}(z)$ if
$\beta$ $\geq$ $\frac{e}{e-1} \approx 1.581976$.
\item $\Theta_{\beta}(z,p(z)) \prec \phi_{c}(z)$ if
$\beta$ $\geq$ $\frac{5e^{e-1}}{3(e^{e-1}-1)} \approx 2.030970$.
\end{enumerate}
The bounds in each case are sharp.
\end{theorem}

\begin{proof}
\begin{enumerate}[(a)]
\item Define the analytic function  $q_{\beta}:\overline{\mathbb{D}}\rightarrow \mathbb{C}$ by
\[q_{\beta}(z)=1+\frac{1}{\beta}\left(\frac{4z}{3}+\frac{z^{2}}{3}\right)\]
It is easy to see that the  function $q_{\beta}$ satisfies the differential equation
$\beta zq'(z) = \phi_{c}(z)-1$.
Proceeding as similar lines in Theorem \ref{first thm for bell}, the required subordination holds if and only if,
\begin{equation}\label{eq1}
e^{e^{-1}-1} < q_{\beta}(-1) < q_{\beta}(1) < e^{e-1}.
\end{equation}
Simplifying the condition \eqref{eq1}, we obtain the inequalities
\[\beta\geq \frac{1}{1-e^{(e^{-1}-1)}}=\beta_{1} \quad \text{and}\quad\beta\geq \frac{5 e}{3 \left(e^e-e\right)}=\beta_{2}.\] Thus, the required subordination holds if
$\beta \geq \max\{\beta_{1},\beta_{2}\}=\beta_{1}$.
	
\item Define the analytic function $\breve{q}_{\beta}(z)$ by,
\[\breve{q}_{\beta}(z)=\exp\left(\frac{1}{\beta}\left(\frac{4z}{3}+\frac{z^{2}}{3}\right)\right)\]
which is a solution of the equation
\[\frac{d\breve{q}_{\beta}'(z)}{dz}=\frac{2(2+z)}{3 \beta}\breve{q}_{\beta}(z).\]
Proceeding as similar lines in Theorem \ref{th1}, the subordination $p(z) \prec e^{e^{z}-1}$ holds if
$\beta \geq \max\{\breve{\beta}_{1},\breve{\beta}_{2}\}$,
where
\[\breve{\beta}_{1}=\frac{e}{e-1} \text{ and } \breve{\beta}_{2}=\frac{5}{3(e-1)}\]
are obtained from the inequalities
$\breve{q}_{\beta}(-1) > e^{e^{-1}-1}$ and $\breve{q}_{\beta}(1) < e^{e-1}$ respectively.
		
\item The differential equation
\[\frac{d\hat{q}_{\beta}'(z)}{dz}=\frac{2(2+z)}{3 \beta}\hat{q}_{\beta}^2(z)\]
has an analytic solution
\[\hat{q}_{\beta}(z)=\left(1-\frac{1}{\beta}\left(\frac{4z}{3}+\frac{z^{2}}{3}\right)\right)^{-1}\]
in $\mathbb{D}$.
Therefore, proceeding as in Theorem \ref{th2},
the required subordination $p(z) \prec e^{e^{z}-1}$ holds if $\beta \geq \max\{\hat{\beta}_{1},\hat{\beta}_{2}\}=\hat{\beta}_{2}$,
where
\[\hat{\beta}_{1}= \frac{e^{\frac{1}{e}-1}}{1-e^{\frac{1}{e}-1}}\text{ and }\hat{\beta}_{2}=\frac{5e^{e-1}}{3(e^{e-1}-1)}.\]
\end{enumerate}
\qedhere
\end{proof}
\begin{corollary}
Let $f \in \mathcal{A}$ be given by $f(z)=z+\sum\limits_{n=2}^{\infty}a_{n}z^{n}$. If one of the following subordinations holds
\begin{enumerate}[(a)]
	\item $1+ \beta \frac{zf'(z)}{f(z)} \mathfrak{M}(z) \prec \phi_{c}(z)$ for
	$\beta$ $\geq$ $\frac{1}{1-e^{(e^{-1}-1)}}$,
	\item $1+ \beta \mathfrak{M}(z) \prec \phi_{c}(z)$ for
	$\beta$ $\geq$ $\frac{e}{e-1}$,
	\item $1+ \beta  \left(\frac{zf'(z)}{f(z)}\right)^{-1}\mathfrak{M}(z) \prec \phi_{c}(z)$ for
	$\beta$ $\geq$ $\frac{5e^{e-1}}{(3(e^{e-1}-1))}$,
\end{enumerate}
then $f \in \mathcal{S}^{*}_{B}$, where $\mathfrak{M}(z)$ is defined in Corollary \ref{Cor1}.
\end{corollary}

The next results provide best possible bound on $\beta$ so that the subordination $1+\beta z p'(z)/p^{j}(z) \prec \phi_{c}(z),\,\phi_{0}(z)  (j=0,1,2)$ implies the subordination $p(z) \prec \phi_{SG}(z)$. Proofs of the following results are omitted as similar to the previous Theorem \ref{thm1}.
		
\begin{theorem}\label{thm SG phi_{0}}
Let $p$ be an analytic function in $\mathbb{D}$ with $p(0)=1$.
Then the following subordinations hold for $p(z) \prec \phi_{SG}(z):={2}/{(1+e^{-z})}$.

\begin{enumerate}[(a)]
\item $\Psi_{\beta}(z,p(z)) \prec \phi_{0}(z)$ if $\beta$ $\geq$ $\frac{(e+1)(1-\sqrt{2}-2\log(2-\sqrt{2}))}{e-1} \approx 1.418226$.
\item $\Lambda_{\beta}(z,p(z)) \prec \phi_{0}(z)$ if $\beta$ $\geq$ $\frac{1-\sqrt{2}-2\log(2-\sqrt{2})}{1+\log2-\log(1+e)} \approx 1.725221$.
\item $\Theta_{\beta}(z,p(z)) \prec \phi_{0}(z)$ if $\beta$ $\geq$ $\frac{2e(1-\sqrt{2}-2\log(2-\sqrt{2}))}{e-1} \approx 2.073612$.
\end{enumerate}
The bounds on $\beta$ in each case are sharp.
\end{theorem}

\begin{theorem}
Let \textit{p} be an analytic function in $\mathbb{D}$ which satisfies $p(0)=1$.
Then each of the following subordination is sufficient for $p(z) \prec \phi_{SG}(z)$.

\begin{enumerate}[(a)]
\item $\Psi_{\beta}(z,p(z)) \prec\phi_{c}(z)$ if
	$\beta$ $\geq$ $\frac{5(e+1)}{3(e-1)} \approx 3.60659$.
\item $\Lambda_{\beta}(z,p(z)) \prec \phi_{c}(z)$ if
	$\beta$ $\geq$ $\frac{5}{3(1+\log2-\log(1+e))} \approx 4.387286$.
\item $\Theta_{\beta}(z,p(z)) \prec \phi_{c}(z)$ if
	$\beta$ $\geq$ $\frac{10e}{3(e-1)} \approx 5.27326$.
\end{enumerate}
The bounds on $\beta$ in each case are sharp.
\end{theorem}

\end{document}